\documentclass[12pt,a4paper]{amsart}

\usepackage{amsfonts,amsmath,amsthm,  amssymb, amscd}

\newtheorem{theorem}{Theorem}
\newtheorem{corollary}{Corollary}

\newtheorem{lemma}{Lemma}

\newtheorem{remark}{Remark}

\thanks{
Supported by  UAEU grants:  UPAR  G00001922 and  GG00002160}

\title[Completely simple rings]
{Completely simple endomorphism rings \\of modules}
\author{V.~Bovdi, M.~Salim, Mihail Ursul}

\address{Department of Math. Sciences\\
UAE University\\
Al-Ain\\
United Arab Emirates}
\email{vbovdi@gmail.com}
\email{m.salim@uaeu.ac.ae}

\address{Department of Mathematics and Computer Science,\\ University of Technology, Lae,\\
Papua New Guinea}
\email{mihail.ursul@gmail.com}

\begin{document}
\maketitle
%\large

\begin{abstract}
It is proved that if $A_p$ is a countable elementary abelian $p$-group, then:
(i) The ring ${\rm End}\,(A_p)$ does not admit a nondiscrete locally compact ring topology.
(ii) Under (CH) the simple ring  ${\rm End}\,(A_p)/I $, where $I$ is the ideal of ${\rm End}\,(A_p)$ consisting of all endomorphisms with finite images,  does not admit a nondiscrete locally compact ring topology.
(iii) The finite topology  on  ${\rm End}\,(A_p)$ is the only second metrizable ring topology on it.
Moreover, a characterization of completely simple endomorphism rings of the endomorphism rings of modules over  commutative rings  is also obtained.
\end{abstract}

\section{Introduction}
The notion of associative simple ring can be extended for associative  topological rings in several  ways:
\begin{itemize}
\item[{(i)}]
simple abstract ring endowed with a nondiscrete ring topology (for example,  the classification of nondiscrete locally compact division rings, see \cite{Pontrjagin_book}, Chapter IV and  \cite{Bourbaki_II, Jacobson_2,  Jacobson_3}; we refer to  some historical notes about  locally compact division rings to  \cite{Warner_book_II});
\item[{(ii)}]
 topological ring without nontrivial closed ideals (for example,  see \cite{Leptin, Zelinsky_I}).
\item[{(iii)}]topological ring $R$ with the property that if $f:R\to S$ is a continuous homomorphism in a topological ring $S$, then either $f=0$ or $f$ is a topological embedding of $R$ into $S$ (see \cite{Mutylin}).
\end{itemize}

In all cases it is assumed that multiplication is not trivial.

I.~Kaplansky has mentioned  (see \cite{Kaplansky_selected}, p.56) that classification of locally compact simple rings in characteristic $p>0$  is difficult. He proved that every simple nondiscrete locally  compact simple torsion-free ring is a matrix ring over a locally  compact division  ring. However in \cite{Skornjakov} (see also \cite{Warner_book}) has been constructed a nondiscrete locally  compact simple ring of positive characteristic which is not a  matrix ring over a division ring. Thereby the program of classification of nondiscrete locally compact simple rings was finished. Nevertheless it is interesting  to look for new examples of locally compact simple rings.

If $A_p$ is a countable elementary abelian $p$-group and $I$ is an ideal of  ${\rm End}\,(A_p)$ consisting of endomorphisms with finite images, then the factor ring ${\rm End}\,(A_p)/I$ is a simple von Neumann regular ring. We prove that under (CH) this ring does not admit a nondiscrete locally compact ring topology.

\bigskip

S.~Ulam (see \cite{Scottish_Book}, Problem 96, p.181)  posed  the following problem: "{\it Can the group $S_\infty$ of all permutations of integers so metrized that the group operation (composition of permutations) is a continuous function and the set $S_\infty$ becomes, under this metric, a compact space? (locally compact?)}". E.D.~Gaughan (see \cite{Gaughan}) has solved this problem in the negative.

\bigskip

In  \S\ref{P:1}   we  study an analogous problem for  the endomorphism ring of a countable elementary abelian $p$-group, namely: "{\it Does the endomorphism ring ${\rm End}(A_p)$ of a countable elementary $p$-group $A_p$ admit a nondiscrete locally compact ring topology?}".  Similarly to Ulam's problem we obtain a negative answer to this question. Moreover, we prove that $\mathcal{T}_{fin}$ is the only ring topology $\mathcal{T}$ on ${\rm End}\,(A_p)$ such that $({\rm End}\,(A_p), \mathcal{T})$ is complete and second metrizable.

\bigskip

We classify  in \S4   the completely  simple rings $({\rm End}\,(M), \mathcal{T}_{fin})$  of  vector spaces $M$ over division rings.
Corollary \ref{C:2} gives a characterization of semisimple left linearly compact minimal rings.  It should be mentioned that Corollary \ref{C:2} is related to a result from \cite{Arnautov_Ursul} stating that any semisimple ring admits at most one linearly compact topology.

\bigskip

Furthermore,  we obtain in \S5 a description of  completely simple rings of the form $({\rm End}\,(M_R), \mathcal{T}_{fin})$ of modules $M$ over a commutative ring $R$. We extend the result of \cite{Ware_Zelmanowitz} to topological rings $({\rm End}\,(M_R), \mathcal{T}_{fin})$.

\section{Notation, Conventions and Preliminary Results}

Rings are assumed to be associative, not necessarily with identity. Topological spaces are assumed to be completely regular.
The {\it weight} (see \cite{Engelking}, p.12) of the space $X$ is denoted by $w(X)$. The {\it pseudocharacter} of a point $x\in X$  (see \cite{Engelking}, p.135) is the smallest cardinal of the form $|\mathcal{U}|$, where  $\mathcal{U}$  is a family of open subsets of $X$ such that $\cap\mathcal{U}=\{x\}$. The closure of a subset $A$ of the topological space $X$ is denoted by $\overline{A}$ and the interior by ${\rm Int}(A)$ (see \cite{Engelking}, p.14). A topological space $X$ is called a {\it Baire  space} (see \cite{Engelking}, p.198) if for each sequence $\{X_1,X_2,\ldots\}$ of open dense subsets of $X$ the intersection $\cap_{i=1}^\infty  G_i$ is a dense set.

An abelian group $A$ is called {\it elementary abelian $p$-group} ($p$ prime) if $pa=0$ for all $a\in A$.  It is well-known  that   $A$  is a direct sum of copies of the cyclic group $Z(p)$. The subring of a ring $R$ generated by a subset $S$, is denoted by $\langle S\rangle$. A ring $R$ is called {\it locally finite} if every its finite subset is contained in a finite subring. A topological  ring $(R,\mathcal T)$  is called {\it metrizable} if its underlying additive group satisfies the first axiom of countability. A ring $R$ with $1$ is called {\it Dedekind-finite} if each equality $xy=1$ implies $yx=1$. It is well-known that every finite ring with identity is Dedekind-finite. Since every compact ring with
identity is a subdirect product of finite rings, it follows that every compact ring with identity is Dedekind-finite.
If $A\subseteq  R$, then  ${\rm Ann}_l(A):=\{x\in R\mid xA=0\}$. If $X,Y$ are the subsets of $R$, then $X\cdot Y:=\{xy\mid x\in X, y\in Y\}$.  A topological ring $R$ is called {\it compactly generated} (see \cite{Ursul}, Chapter II) if there exists a compact subset $K$ such that $R=\langle K\rangle$. If $(R,\mathcal T)$ is a topological ring and $I$ is an ideal of $R$, then the quotient topology of the factor ring $R/I$ is denoted by $\mathcal T/I$. If $K$ is a  subset of  an abelian group $A$, then  set
\[
T(K)=\{\alpha\in {\rm End}\,(A)\mid \alpha(K)=0 \}.
\]
When $K$ runs over all finite subsets of $A$, the family $\{T(K)\}$ defines a ring topology ${\mathcal{T}_{fin}}$ on
${\rm End}\,(A)$. This topology is called the {\it finite topology}.

\begin{lemma} \label{LM:1}
For any abelian group $A$ the ring $({\rm End }\, A,\mathcal T_{fin})$ is complete.
\end{lemma}
\begin{proof}
See \cite{Ursul}, Theorem 19.2.
\end{proof}

\begin{lemma}\label{LM:2}
(Cauchy's criterion) In a Hausdorff complete commutative group $G$, in order that a family $(x_\alpha)_{\alpha\in\Omega}$ should be summable it is necessary and sufficient that, for each neighborhood $V$ of zero in $G$, there is a finite subset $\Omega_0$ of $\Omega$ such that $\Sigma_{\alpha\in K}x_\alpha\in V$ for all finite subsets $K$ of $\Omega$ which do not meet $\Omega$.
\end{lemma}
\begin{proof}
See \cite{Bourbaki_III}, p.263.
\end{proof}

\begin{lemma}\label{LM:3}
If $(x_\alpha)_{\alpha\in\Omega}$ is a summable subset in $({\rm End }\, A,\mathcal T_{fin})$ then  every subset $\Delta$ of $\Omega$ the family $(x_\beta)_{\beta\in\Delta}$ is summable.
\end{lemma}
\begin{proof}
Let $V$ be a neighborhood of zero of $({\rm End }\, A,\mathcal T_{fin})$. We can consider without loss of generality that $V$ is a left ideal of ${\rm End }\, A$. There exists a finite subset $\Omega_0$ of $\Omega$ such that $\Sigma_{\alpha\in K}x_\alpha\in V$ for every finite subset $K$ of $\Omega$ for which $K\cap\Omega_0=\varnothing$. Let $F$ be a finite subset of $\Delta$ such that $F\cap(\Omega_0\cap \Delta)=\varnothing$. If $\alpha\in F$, then $\alpha\notin\Omega_0$, hence $\Sigma_{\alpha\in F}x_\alpha\in V$. By Cauchy's criterion the family $(x_\beta)_{\beta\in\Delta}$ is summable.
\end{proof}

A topological ring $(R,\mathcal{T})$ is called {\it minimal} (see, for instance, \cite{Dikranjan})  if there is no ring topology $\mathcal{U}$ such that $\mathcal{U}\leq \mathcal{T}$
and $\mathcal{U}\not= \mathcal{T}$. A topological ring $(R,\mathcal{T})$ is called {\it simple} if  $R$ is simple as a ring without topology. A topological ring $(R,\mathcal{T})$ is called {\it weakly simple}  if  $R^2\neq 0$ and  every its closed ideal is either  $0$ or $R$.
A topological ring $(R,\mathcal{T})$ is called {\it completely simple} (see  \cite{Mutylin})  if  $R^2\neq 0$ and  for every continuous homomorphism   $f: (R,\mathcal{T})\to(S,\mathcal{U})$ in a topological ring $(S,\mathcal U)$ either  $\ker(f)=R$ or $f$ is a homeomorphism of $(R,\mathcal T)$ on ${\rm Im}(f)$. Equivalently,   $R^2\not= 0$ and   $(R,\mathcal{T})$ is weakly simple and minimal.
Let $M$ be a unitary right $R$-module over   a commutative ring $R$  with $1$.
The  module $M$ is called \emph{divisible} if $Mr=M$ for every $0\neq r\in R$.
A  right  $R$-module $M$ is called {\it faithful} if $Mr=0$ implies $r=0$ ($r\in R$).
A   right  $R$-module $M$ is called {\it torsion-free} if $mr=0$ implies that either $m=0$ or $r=0$, where $m\in M$ and $r\in R$.
Recall that a submodule $N$ of an $R$-module $M$ is called {\it fully invariant} $\alpha(N)\subseteq N$ for every endomorphism $\alpha$ of $M_R$.
We use in the sequel the notion and results from  the books \cite{Engelking, Ursul}.

\begin{remark}
If $R$ is a von Neumann regular ring, then $R^2=R$.
\end{remark}

\begin{lemma}\label{LL:1}
An ideal $I$ of a von Neumann regular ring is von Neumann regular.
\end{lemma}

\begin{proof}
Let $i\in I$. Thus there exists $x\in R$ such that $ixi=i$. It follows that $ixixi=i $ and $xix\in I$.
\end{proof}

\begin{corollary}If $I$ an ideal of  a von Neumann regular ring $R$, then any ideal $H$ of $I$ is an ideal of $R$, too.
\end{corollary}
\begin{proof}
$RH=RH^2\subseteq IH\subseteq H$. Similarly, $HR\subseteq H$.
\end{proof}

If $A_p$ is a $p$-elementary countable group, then
\[
I=\{\alpha\in {\rm End}\,(A_p)\vert\vert  {\rm Im}(\alpha)\vert<\aleph_0\}.
\]
Fix a linear basis $\{v_i\mid i\in \mathbb N\}$ of $A_p$ over the field $\mathbb F_p$. Using this fixed  basis, we define  the map  $e_i: A \to A$ such that
\[
e_i(v_j)=\delta_{ij}v_j, \qquad (i, j\in\mathbb N)
\]
where $\delta_{ij}$ is the Kronecker delta.

\begin{lemma}\label{LL:2} We have for ${\rm End}\,(A_p)$:
\begin{itemize}
\item[(i)] $I$ is a von Neumann regular ring.
\item[(ii)] $I$ is a simple ring.
\item[(iii)] The factor ring ${\rm End}(A_p)/I$ is  simple  von Neumann regular.
\item[(iv)] $I$ is a locally finite ring.
\end{itemize}
\end{lemma}
\begin{proof}

(i):\quad  It is well-known that  ${\rm End}(A_p)$ is regular (see \cite{Lam}, Theorem 4.27, p.63). Then  $I$ is von Neumann regular by Lemma \ref{LL:1}.

(ii), (iii):\quad  The ideal  $I$ is the only nontrivial ideal of   ${\rm End}(A_p)$ (see \cite{Jacobson}, \S17, Theorem 1,  p.93).  This means that  ${\rm End}(A)/I$  is simple. It is  regular by the part (i).

(iv)\quad  Since $I$ is simple (see \cite{Jacobson},  \S 12, Proposition 1),  it suffices to show that $I$ contains a nonzero locally finite right ideal.

Let us show  that the left ideal $Ie_1$ of $I$ is locally finite as a ring  (equivalently, as a $\mathbb F_p$-algebra). We have $0\neq e_1\in Ie_1$.  If $H$ is the left annihilator of $Ie_1$, then, obviously, $H$ is a locally finite ring, hence it is locally finite as a $\mathbb F_p$-algebra. We claim that $Ie_1/H$ is finite.
Indeed, define $\beta_n\in H$ $(n\ge 2)$ in the following way
\[
\beta_n(v_i)=
\begin{cases} v_n,\quad\text{for}\quad i=1;\\
0,\qquad  \text{for}\quad i\neq 1.\\
\end{cases}
\]
We claim that $Ie_1=\mathbb F_pe_1+\Sigma_{n=2}^\infty \mathbb F_p\beta_n$.

For, let $\alpha\in I$. Then $\alpha(v_1)=r_1v_1+\cdots+r_nv_n$, where $r_i\in\mathbb F_p$ and $n\in\mathbb N$.  Thus
\[
\begin{split}
\alpha e_1(v_1)& =r_1e_1(v_1)+r_2\beta_2(v_1)+\cdots+r_n\beta_n(v_1)\\
&=(r_1e_1+r_2\beta_2+\cdots+r_n\beta_n)(v_1);\\
\alpha e_1(v_j)&=(r_1e_1+r_2\beta_2+\cdots+r_n\beta_n)(v_j)\qquad (j\neq 1).
\end{split}
\]
This yields
\[
\alpha e_1=r_1e_1+r_2\beta_2+\cdots+r_n\beta_n \qquad
\]
and so $Ie_1=\mathbb F_pe_1+\Sigma_{n=2}^\infty \mathbb F_p\beta_n$.

In particular,
$Ie_1=\mathbb F_pe_1+H$, and so $H$ has a finite index in $Ie_1$. Clearly,  $Ie_1$ is a locally finite $\mathbb F_p$-algebra (see \cite{Jacobson},  Proposition 1,  p.241) and  $I$ is a locally finite $\mathbb F_p$-algebra (see \cite{Jacobson}, Proposition 2, p.242).\end{proof}

The next result  can be deduced from Lemma 36.11 of \cite{Ursul}.
\begin{lemma}\label{LL:3}
Let $A$ be a  locally compact, compactly generated, and  totally disconnected ring. If  $A$ contains a dense   locally finite subring $B$, then $A$ is compact.
\end{lemma}
\begin{proof}
Let $A=\langle V\rangle$, where $V$ is a compact symmetric neighborhood of zero. Since $V$ is compact, the subset $V+V+V\cdot V$ also is compact. Since $B$ is dense, $A=B+V$. By compactness of $V+V+V\cdot V$ there exists a 	 finite subset $H\subseteq  B$ such that $V+V+V\cdot V\subseteq H+V$. Since $B$ is a locally finite ring, we can assume without loss of generality that $H$ is a subring. Let  $H\setminus\{0\}=\{h_1,\ldots,h_k\}$.
The subset
\[
H+h_1V+\cdots+h_kV+V
\]
 is an open subgroup of $R(+)$. Indeed, this subset is symmetric and
\[
\begin{split}
(H+h_1V+\cdots+h_kV+V)+(H+h_1V+\cdots+h_kV+V)&\\
\subseteq H+h_1(V+V)+\cdots+h_k(V+V)+V+V&\\
\subseteq H+h_1V+\cdots+h_kV+V&.
\end{split}
\]
We  prove by induction on $m$ that
\[
V^{[m]}\subseteq H+h_1V+\cdots+h_kV+V,  \qquad (m\in\mathbb{N})
\]
where $V^{[1]}=V$ and $V^{[m]}=V^{[m-1]}\cdot V$ for all $m$.

The inclusion is obvious for $m=1$.

Assume that the assertion has been proved for $m\ge1$. Clearly,
\[
\begin{split}
V^{[m+1]}=V^{[m]}\cdot V\subseteq H\cdot V+h_1(V\cdot V)+\cdots + h_k(V\cdot V)+V\cdot V& \subseteq \\
h_1V+\cdots+ h_kV+h_1(H+V)+\cdots+ h_k(H+V)+H+V& \subseteq\\
H+h_1V+\cdots+h_kV+V&.
\end{split}
\]
Consequently,   $A=H+h_1V+\cdots+h_kV+V$, therefore $A$ is compact.
\end{proof}

An element $x$ of a topological ring is called {\it discrete} if there exists a neighborhood $V$ of zero such that $xV=0$ (i.e., the right annihilator of $x$ is open).

\begin{lemma}\label{LL:4}
The set of all discrete elements of  a topological ring is  an ideal. A simple ring with identity does not contain nonzero discrete elements.
\end{lemma}

\section{Locally compact ring topologies  on  ${\rm End}\,(A)$ of a countable elementary abelian $p$-group $A$}\label{P:1}

\begin{theorem}\label{T:1}
Let $R$ be a simple,   nondiscrete and  locally compact ring of  ${\rm char}(R)=p>0$ and $1\in R$.
If $V$ is a compact open subring of $R$ and $\{e_\alpha\vert \alpha\in\Omega\}$ is a set of orthogonal idempotents in $R$, then
\[
|\Omega|\leq w(V).
\]
\end{theorem}
\begin{proof}
The ring $R$ does not contain nonzero discrete elements by Lemma  \ref{LL:4}. Since $R$ is locally compact and ${\rm char}(R)=p$,  it is totally disconnected. Additionally, $R$ has a fundamental system of neighborhoods of zero consisting of compact open subrings by  Lemma 9 of \cite{Kaplansky}.

If $V$ is a compact open subring of $R$, then by continuity of the ring operations  for each $\alpha\in\Omega$ there exists an open ideal $V_\alpha$ of $V$ such that $e_\alpha V_\alpha\subseteq V$.  Clearly, there exists   $y_\alpha\in V_\alpha$ for which   $e_\alpha y_\alpha\neq 0$ since $R$ has no nonzero discrete elements.

\noindent
We claim that hold the following two properties:

(i) $e_\alpha y_\alpha\not\in \overline{\{e_\beta y_\beta\vert \beta\neq \alpha\}}$ for each $\alpha\in\Omega$;

(ii) the set    $X=\{e_\alpha y_\alpha\vert \alpha\in \Omega\}$ is a discrete subspace of $V$.

Indeed, if were  $e_\alpha y_\alpha\in \overline{\{e_\beta y_\beta\vert \beta\neq \alpha\}}$ for some $\alpha\in \Omega$, then
\[
\begin{split}
e_\alpha y_\alpha=e_\alpha e_\alpha y_\alpha&\in e_\alpha\overline{\{e_\beta y_\beta\vert \beta\neq \alpha\}}\\
&\subseteq \overline{\{e_\alpha e_\beta y_\beta\vert \beta\neq \alpha\}}\\
&=\{0\},
\end{split}
\]
so  $e_\alpha y_\alpha=0$, a contradiction. The part (i) is proved.

(ii) Now,   for  each $\alpha\in\Omega$  we have  $V\setminus\overline{\{e_\beta y_\beta\vert \beta\neq \alpha\}}$ is open and, consequently,
\[
(V\setminus\overline{\{e_\beta y_\beta\vert \beta\neq \alpha\}})\cap X=\{e_\alpha y_\alpha\},
\]
by (i). Therefore the point $e_\alpha y_\alpha(\alpha\in\Omega)$ of $X$ is isolated. In other words, the subspace $X$ of $V$ is discrete.

Since $X$ is discrete,      $|\Omega|=|X|=w(X)\leq w(V)$  (see \cite{Arkhangelskii_Ponomarev}, Exercises 98-99, p.72).
\end{proof}

\begin{theorem}\label{T:2}
Let $A_p$ be a countable elementary abelian $p$-group.
Then the ring \[
I=\{\alpha\in {\rm End}\,(A_p)\vert\vert  {\rm Im}(\alpha)\vert<\aleph_0\}
\]
 does not admit a nondiscrete  ring topology $\mathcal U$
such that $(I,\mathcal U)$ is a Baire space.
\end{theorem}
\begin{proof}

Put $S_n=\{\alpha\in I\mid  \alpha(A)\subseteq {\mathbb{F}_p} v_1+\cdots+\mathbb F_pv_n\}$, where $n\in\mathbb{N}$.
Clearly  $I=\cup_{n\in\mathbb{N}}S_n$ and
\[
S_n=\{\alpha\in I \mid e_i\alpha=0 \quad  \text{for} \quad  i> n\}={\rm Ann}_r\big(\{e_k\mid k>n\}\big).
\]
This yields that   $S_n$ is closed due continuity of the ring operations.

Since $I$ is a Baire space, there exists $n\in\mathbb{N}$ such that ${\rm Int}(S_n)\neq\emptyset$, hence $S_n$ is an open subgroup.

Set $\beta\in I$ such that
\[
\beta(v_i)=
\begin{cases} v_{n+i},\quad\text{for}\quad i=1,\ldots,n;\\
0\quad,  \quad \text{for}\quad i> n.\\
\end{cases}
\]
Let $W\subseteq S_n$ be a neighborhood of zero of $(I,\mathcal U)$ such that $\beta W\subseteq S_n$. If $w\in W\setminus 0$,  then there exist $a\in A$ and $ r_1,\ldots, r_n\in\mathbb F_p$ such that
\[
0\neq w(a)=\sum_{i=1}^n r_iv_i\qquad
\text{and} \qquad \beta( w(a))=\sum_{i=1}^n  r_iv_{n+i}.
\]
There exists  $j\in 1,\ldots, n$ such that $r_j\neq 0$. Then
\[
e_{n+j}\beta w(a)=r_jv_{n+j}\neq 0,
\]
 hence $e_{n+j}\beta w\neq 0$ and so  $\beta w\not\in S_n$, a contradiction.\end{proof}

\begin{corollary}\label{C:1}
Under the notation of Theorem \ref{T:2} the ring  $I$ does not admit a nondiscrete locally compact ring topology.
\end{corollary}
\begin{proof}
This follows from the fact that  each  locally compact space is a Baire space  (see \cite{Bourbaki}, Theorem 1, p.117).
\end{proof}

Our main result is the following.

\begin{theorem}\label{T:3}
The endomorphism ring ${\rm End}\,(A_p)$ of a countable elementary abelian $p$-group $A_p$ does not admit a nondiscrete locally compact ring topology.
\end{theorem}
\begin{proof}

We will use the notation and results from section 2.
Denote $R= {\rm End}\,(A_p)$. Assume on the contrary that there exists on $R$ a nondiscrete locally compact ring topology $\mathcal T$.

\noindent
\underline{Fact 1.} The ring $(R,\mathcal T)$ has a fundamental system of neighborhoods of zero consisting of compact open subrings.

Since  the additive group of $R$  has exponent $p$, it  is totally
disconnected (this follows from \cite{Hewitt_Ross}, Theorem 9.14, p.95).
By   I.~Kaplansky's result (see \cite{Kaplansky}, Lemma 9) $(R,\mathcal T)$ has a  fundamental system of neighborhoods of zero consisting of compact open subrings.

\noindent
\underline{Fact 2.} The group  $Re_i$
is countable for each $i\in\mathbb N$.

\noindent
We claim  that $Re_i$ is infinite. Indeed, for each
$j\in \mathbb{N}$  put $\beta_j\in R$ such that
\[
\beta_j(v_k)=
\begin{cases}
v_j, \quad \text{for}\quad k=i;\\
0, \qquad \text{for}\quad  k\not=i.\\
\end{cases}
\]
If $j\not=s$, then $\beta_je_i(v_i)=\beta_j(v_i)=v_j$ and
$\beta_se_i(v_i)=\beta_s(v_i)=v_s$, hence $\beta_je_i\not=\beta_se_i$, so $Re_i$
is infinite.

The ring $Re_i$ is countable. Indeed, consider the mapping $\psi: Re_i\to A_p^{\mathbb{F}_pv_i}$, where
\[
\psi(\alpha e_i)(rv_i)=\alpha(rv_i)\qquad\text{for all}\qquad  r\in \mathbb{F}_p.
\]
If $\alpha e_i\not=\beta e_i$ ($\alpha, \beta\in R$), then there exists $x=\sum_j r_jv_j\in A_p$ such that
$\alpha e_i(x)\not=\beta e_i(x)$, hence,  $\alpha (r_iv_i)\not=\beta(r_iv_i)$. Thus
\[
\psi(\alpha e_i)(r_iv_i)=\alpha (r_iv_i)\not=\beta(r_iv_i)=\psi(\beta e_i)(r_iv_i).
\]
The latter means that $\psi$ is an injective mapping of  $Re_i$ into $A^{\mathbb{F}_pv_i}$ .
Since $A^{\mathbb{F}_pv_i}$ is countable, $Re_i$ is countable, too.

\noindent
\underline{Fact 3}. $I$ is a closed ideal of $R$.
 We claim that $I$ is not dense in  $(R,\mathcal T)$.
Assume the contrary.
Since $I$  is locally  finite and is a maximal ideal,   $(R,\mathcal T)$
is topologically locally finite by Lemma  \ref{LL:3}. The ring $R$ contains two elements
$x, y$ such that $xy=1$ and $yx\neq 1$. The subring $\overline{\langle x,y\rangle}$
is compact, hence Dedekind-finite, a contradiction. We obtained that
$(R/I,\mathcal T/I)$ is a nondiscrete  metrizable locally compact ring.

\noindent
\underline{Fact 4.} $I$ is a discrete ideal of $R$.

\noindent
This follows from Theorem 2.

\noindent
\underline{Fact  5.} $Re_i$ is a discrete left ideal of $R$ for every $i\in\mathbb N$.

\noindent
Indeed, $Re_i\subseteq I$ and $I$ is discrete by Fact 4 for every $i\in\mathbb N$.

\noindent
\underline{Fact  6.}  ${\rm Ann}_l(e_i)$ is  open in $R$ for every $i\in\mathbb N$.

\noindent
 Indeed, the group homomorphism
$q: R\to Re_i, r\mapsto re_i,$ is continuous. Since $Re_i$ is discrete $q^{-1}(0)= {\rm Ann}_l(e_i)$
is open.

\noindent
\underline{Fact 7.}   $\cap_i{\rm Ann}_l(e_i)=0$.

\noindent
Obvious.

\noindent
\underline{Fact 8.} $\mathcal{T}\geq \mathcal{T}_{fin}$.

We notice that ${\rm Ann}_l(e_i)=T(\{v_i\})$ for every $i\in\mathbb N$. For, if $\alpha e_i=0$,  then $\alpha(v_i)=\alpha e_i(v_i)=0$, i.e., $\alpha\in T(\{v_i\})$.
Conversely, if  $\alpha\in T(\{v_i\})$, then $\alpha e_i(v_i)=\alpha(v_i)=0$. If $j\not=i$ then $\alpha e_i(v_j)=0$. Therefore
$\alpha e_i=0$. Furthermore
\[
T(\{v_1,\ldots,v_n\})=\cap_{i=1}^n T(\{v_i\})=\cap_{i=1}^n {\rm Ann}_l(e_i)\in \mathcal{T}
\] for each $n\in\mathbb N.$\\
Since the family  $\{T(\{v_1,\ldots,v_n\})\}$ forms a fundamental system of  neighborhoods
of zero of $(R,\mathcal T_{fin})$, we get that $\mathcal T_{fin}\leq \mathcal{T}$.

\noindent
\underline{Fact 9.} The ring $(R,\mathcal T)$  is metrizable.

Since $\cap_{i\in\mathbb{N}}{\rm Ann}_l(e_i)=0$, the pseudocharacter of $(R,\mathcal T)$ is $\aleph_0$.
If $V$ is a compact open subring of $(R,\mathcal T)$ (see Fact 1), then the pseudocharacter of $V$ also is $\aleph_0$.
However in every compact space the pseudocharacter of a point coincides with its character. Therefore $(R,\mathcal T)$  is metrizable.

\noindent
\underline{Fact 10.}
$(R/I,\mathcal T/I)$ has an open compact subring.

Indeed, it is well-known (see \cite{Kaplansky}) that every totally disconnected ring has a fundamental system of
neighborhood of zero consisting of compact open subrings.
Henceforth  $V$ is  a fixed open compact  subring of $(R/I,\mathcal T/I)$.

\noindent
\underline{Fact 11.}  $R/I$ contains a family of  orthogonal idempotents of cardinality $2^{\aleph_0}$ .

Indeed, the family $\{e_i\}_{i\in\mathbb{N}}$ of idempotents of  $(R, \mathcal T_{fin})$ is summable in  and $1_A=\Sigma_{n\in\mathbb{N}}e_n$, where $1_A$ is the identity of $R$.

The first ordinal  number of cardinality $\mathfrak{c}$ of continuum is denoted by $\omega(\mathfrak{c})$. Let $\{\mathbb{N}(\alpha)\vert \alpha< \omega(\mathfrak{c})\}$  be a family of infinite almost disjoint subsets of $\mathbb{N}$ (see \cite{Engelking}, Example 3.6.18, p.175--176).
Put  $f_{\mathbb{N}(\alpha)}=\Sigma_{i\in \mathbb{N}(\alpha)}e_i$ for every $\alpha<\omega(\mathfrak{c})$. The element $f_{\mathbb{N}(\alpha)}$ exists by Lemma \ref{LM:3}. Then:
\begin{itemize}
\item[(i)] $f_{\mathbb{N}(\alpha)}\notin I$ for every $\alpha  <\omega(\mathfrak{c})$;
\item[(ii)] $f_{\mathbb{N}(\alpha)}f_{\mathbb{N}(\beta)}\in I$ for each $\alpha,\beta <\omega(\mathfrak{c})$ and $ \alpha\neq \beta$.
\end{itemize}
If  $g_\alpha=f_{\mathbb{N}(\alpha)}+I$ for each $\alpha< \omega(\mathfrak{c})$,  then $\{g_\alpha\vert \alpha <\omega(\mathfrak{c})\}$ is the required system of orthogonal idempotents.

The subring  $V$ is metrizable (by Fact 9). Since $V$ is  compact and $R/I$ is a simple von Neumann regular  ring by Lemma  \ref{LL:2} and  $w(V)\leq \aleph_0$,  we obtain a   contradiction to Theorem \ref{T:1}. \end{proof}

\begin{theorem}\label{T:4}(CH)
Under the notation of Theorem \ref{T:3},  the ring $R/I$ does not admit a nondiscrete locally compact ring topology.
\end{theorem}
\begin{proof}
Assume on the contrary that  the factor ring $R/I$ admits  a nondiscrete locally compact ring topology $\mathcal{T}$, so $(R/I, \mathcal{T})$ contains an open compact subring $V$. Since the cardinality of $R/I$ is continuum and $V$ is infinite, the power of $V$ is continuum. Since we have assumed (CH), the subring  $V$ is  metrizable, hence second metrizable (see \cite{Hulanicki, Jones}).  However  we have proved in  Theorem \ref{T:3} that the ring  $R/I$ contains a family of orthogonal idempotents of cardinality $\mathfrak{c}$, a  contradiction with Theorem \ref{T:1}.\end{proof}

\begin{theorem}\label{T:5}
The finite topology $\mathcal T_{fin}$ is the only second metrizable ring topology $\mathcal T$ on $R$  for which $(R,\mathcal T_{fin})$ is complete.
\end{theorem}
\begin{proof}
Let $K=\langle F\rangle$, where  $F$ is  a finite subset of $A$. Clearly,  there exists a subgroup $A^\prime$ of $A$ such that $A=K\oplus A^\prime$.
Choose  ${\rm e}_F\in R$ such that ${\rm e}_F\upharpoonright_K={\rm id}_K$ and ${\rm e}_F(A^\prime)=0$.  Clearly
\[
T(K)=R(1-e_F)
\]
and   $\alpha K=0$ if and only if  $\alpha\in R(1-e_F)$. It follows that the family
$\{R(1-e_F)\}$ where  $F$ runs over all finite subset of $A$, forms a fundamental system
of neighborhoods of zero for $(R,\mathcal T_{fin})$.

There exists an injective map of $Re_F$ to ${\rm Hom}(K,A)$, so the left ideal $Re_F$
is countable, due to countability  ${\rm Hom}(K,A)$. Since $e_F^2=e_F$, the Peirce decomposition
\[
R=Re_F\oplus R(1-e_F)
\]
of $R$ with respect to the idempotent $e_F$ is a decomposition of the  topological group $(R,+,\mathcal T)$.  It follows that $Re_F$ is discrete, hence $R(1-e_F)$ is open (in the topology $\mathcal T$). It follows that $\mathcal T\ge\mathcal T_{fin}$, so $\mathcal T=\mathcal T_{fin}$ (see \cite{Freudenthal},  Theorem 30    or \cite{Graev}). \end{proof}

\section{Completely simple topological endomorphism rings\\ of vector spaces}\label{P:2}

\begin{theorem}\label{T:6}
Let $A_F$ be a right vector space over a division ring $F$ and $S={\rm End}\,(A_F)$.
The following conditions are equivalent:
\begin{itemize}
\item[\rm(i)] $(S, \mathcal{T}_{fin})$ is a completely simple
topological ring.
\item[\rm(ii)] $\dim(A_F)=\infty$ or $\dim(A_F)< \infty$ and $F$ does not admit a nondiscrete ring topology.
\end{itemize}
\end{theorem}
\begin{proof}
(i)$\Rightarrow$ (ii): If $A_F$ is finite-dimensional, then $S$ is discrete and isomorphic to
$\rm M(n,F)$, where $n$ is the dimension of $A_F$. Then, obviously, $F$ does not admit a
nondiscrete ring topology.

(ii)$\Rightarrow$ (i):
If $\dim(A_F)= n< \infty$, then  $S\cong \rm M(n,F)$. Since $F$ does not admit nondiscrete ring topologies, the same holds for  $\rm M(n,F)$.

Let $A_F$ be  infinite dimensional. Fix a   basis $\{x_\alpha\}_{\alpha<\tau}$ over $F$, where $\tau$
is an infinite ordinal number. It is well-known  that the topological ring   $(S, \mathcal{T}_{fin})$ is weakly  simple (see \cite{Leptin}, Satz 12, p.258) and  the family $\{T(x_\alpha)\}_{\alpha<\tau}$ is a prebase at
zero for the finite topology $\mathcal T_{fin}$ of $S$.

Assume on the contrary that there exists a Hausdorff ring topology $\mathcal T$,
coarser that $\mathcal T_{fin}$ and different from  it. Let $e_\alpha\in S$ such that $e_\alpha^2=e_\alpha$ and  $e_\alpha(x_\beta)=\delta_{\alpha\beta}x_\alpha$ for each $\alpha<\tau$, where  $\delta_{\alpha\beta}$ is the Kronecker delta.

\noindent
\underline{Fact 1}. $T(x_\alpha)={\rm Ann}_l(e_\alpha)$ for each $\alpha<\tau$.

Indeed, if $p\in T(x_\alpha)$, then $pe_\alpha(x_\alpha)=p(x_\alpha)=0$. If $\beta \neq \alpha$, then $e_\alpha(x_{\beta})=0$, hence $pe_\alpha=0$, i.e. $p\in {\rm Ann}_l(e_\alpha)$.
Conversely, if $pe_\alpha=0$, then $p(x_\alpha)=pe_\alpha(x_\alpha)=0$, i.e.  $p\in T(x_\alpha)$.

\noindent
\underline{Fact 2}. There exists $\alpha_0<\tau$ for which $Se_{\alpha_0}$ is nondiscrete in $(S,\mathcal T)$.

Assume on the contrary that   for every $\alpha<\tau$ there exists a neighborhood $V_\alpha$ of zero of $(S,\mathcal T)$ such that $Se_\alpha\cap V_\alpha=0$. If $U_\alpha$ is a neighborhood of zero of  $(S,\mathcal T)$ such that $U_\alpha e_\alpha\subseteq V_\alpha$, then $U_\alpha e_\alpha=0$, hence ${\rm Ann}_l(e_\alpha)=T(x_\alpha)$ is open in $(S,\mathcal T)$. Hence  $\mathcal{T}_{fin}\leq\mathcal T$ and  $\mathcal T=\mathcal T_{fin}$, a contradiction.

\noindent
\underline{Fact 3}. \quad $(Se_{\alpha_0}\cap V)x_{\alpha_0}\nsubseteq \oplus_{\beta\in K}x_\beta F$

\noindent
for any neighborhood $V$ of of zero of $(S,\mathcal T)$ and any finite subset $K$ of the set $[0,\tau)$ of all ordinal numbers less than $\tau$.

Assume on the contrary that there exists a finite subset $K$ of  $[0,\tau)$ and a neighborhood $V$ of zero of $(S,\mathcal T)$ such that
\begin{equation}\label{E:1}
(Se_{\alpha_0}\cap V)x_{\alpha_0}\subseteq \oplus_{\beta\in K}x_\beta F.
\end{equation}
Fix $\gamma\in [0,\tau)\setminus K$. For each $\beta\in K$ define  $q_\beta\in S$ such that $q_\beta(x_\beta)=x_\gamma$ and $q(x_\delta)=0$ for $\delta\not=\beta$.

Let $V_0$ be a neighborhood of zero of $(S,\mathcal T)$ such that $V_0\subseteq V$ and $q_\beta V_0\subseteq V$ for all $\beta\in K$.
There exists $0\neq h\in Se_{\alpha_0}\cap V_0$ by  Fact 2. It follows that $hx_{\alpha_0}\not=0$ by  Fact 1.
Since $Se_{\alpha_0}\cap V_0\subseteq Se_{\alpha_0}\cap V$, we obtain that  $hx_{\alpha_0}=\Sigma_{\beta\in K}x_\beta f_\beta$, ($f_\beta\in F$)  by (\ref{E:1}). There exists $\beta_0\in K$ such that $f_{\beta_0}\neq 0$ (because $hx_{\alpha_0}\not=0$), so
\[
\textstyle
q_{\beta_0}h=q_{\beta_0}(\Sigma_{\beta\in K}x_\beta f_\beta)=r_{\beta_0}x_\gamma\not\in \oplus_{\beta\in K}x_\beta F,
\]
a contradiction. Therefore  Fact 3 is proved.

Now let $V$ be a  neighborhood of zero of $(S,\mathcal T)$. Pick up    a neighborhood $V_0$ of zero of $(S,\mathcal T)$  such that $V_0\cdot V_0\subseteq V$.
Since $\mathcal{T}\leq\mathcal{T}_{fin}$,   there exists a finite subset $K$ of $[0,\tau)$ such that
\[
T(\{x_\beta\mid \beta\in K\})\subseteq V_0.
\]
We have   $(Se_{\alpha_0}\cap V_0)x_{\alpha_0}\nsubseteq \oplus_{\beta\in K}x_\beta F$ by  Fact 3. It follows that there exists  $q\in Se_{\alpha_0}\cap V_0$ such that
\[
q(x_{\alpha_0})\not\in \oplus_{\beta\in K}x_\beta F.
\]
Clearly,  $q(x_{\alpha_0})\in A_F$, so it  can be written as  $q(x_{\alpha_0})=\sum_{\alpha< \tau}x_{\alpha}f_{\alpha}$, where $f_{\alpha}\in F$ and here exists $\beta_0\not\in K$  such that $f_{\beta_0}\not=0$.

Consider the element   $s\in S$ such that  $s(x_{\beta_0})=x_{\alpha_0}f_{\beta_0}^{-1}$ and $s(x_{\lambda})=0$ for $\lambda\neq\beta_0$. Evidently, $s\in T(K)$, hence
\[
sq\in T(K)\cdot V_0\subseteq V_0\cdot  V_0\subseteq V.
\]
Moreover, $sq(x_{\alpha_0})=s(x_{\beta_0}f_{\beta_0}+\cdots)=x_{\alpha_0}$. Since $q\in Se_{\alpha_0}$ we obtain that $sq(x_\beta)=0$ for $\beta\neq \alpha_0$. Consequently,  $e_{\alpha_0}=sq\in V$ for every neighborhood $V$ of zero of $(S,\mathcal T)$, a contradiction.
\end{proof}
\begin{remark}
The question of existence of a uncountable division ring which does not admit a nondiscrete Hausdorff ring topology is open.
Several results on this topic  can be found in Chapter 5 of \cite{Arnautov_book}.
\end{remark}

\begin{theorem}\label{T:7}
Let  $\prod_{\alpha\in \Omega}R_{\alpha}$ be a family of compact rings with identity. Then the product $(\prod_{\alpha\in\Omega}R_\alpha,\prod_{\alpha\in\Omega}\mathcal T_\alpha)$ is a minimal ring if and only if  every $(R_\alpha,\mathcal T_\alpha)$ is a minimal topological ring. (Here $\prod_{\alpha\in\Omega}\mathcal T_\alpha$ is the product topology on $\prod_{\alpha\in\Omega}R_\alpha$.)
\end{theorem}
\begin{proof}
$\Rightarrow:$ Assume on the contrary that there exists $\beta\in\Omega$ and a ring topology $\mathcal T^\prime$ on $R_\beta$ such that $\mathcal T^\prime\leq\mathcal T_\beta$ and $ \mathcal T^\prime\neq\mathcal T_\beta$. Consider the product topology $\mathcal U$ on $\prod_{\alpha\in\Omega}R_\alpha$, where $R_\alpha$ is endowed with $\mathcal T_\alpha$ when $\alpha\neq \beta$ and $R_\beta$ is endowed with $\mathcal T^\prime$. Obviously, $\mathcal U\leq\prod_{\alpha\in\Omega}\mathcal T_\alpha$ and $\mathcal U\neq\prod_{\alpha\in\Omega}\mathcal T_\alpha$, a contradiction.

\noindent
$\Leftarrow:$
Denote by $\pi_\alpha(\alpha\in\Omega)$ the projection of $\prod_{\alpha\in \Omega}R_{\alpha}$ on $R_\alpha$.  By definition of the product topology, $\prod_{\alpha\in \Omega}\mathcal T_{\alpha}$ is the coarsest topology on $\prod_{\alpha\in \Omega}R_{\alpha}$ for which the projections $\pi_\alpha(\alpha\in\Omega)$ are continuous.

Let $\mathcal U$  be a ring topology on  $\prod_{\alpha\in \Omega}R_{\alpha}$, $\mathcal U\leq\prod_{\alpha\in \Omega}\mathcal T_\alpha$ and $\beta\in \Omega$. Since
\[
\mathcal U\upharpoonright_{R_\beta\times\prod_{\gamma\neq\beta}\{0_\gamma\}}\leq\big(\prod_{\alpha\in\Omega}\mathcal T_\alpha\big)\upharpoonright_{R_\beta\times\prod_{\gamma\neq\beta}\{0_\gamma\}},
\]
it follows that $\mathcal U\upharpoonright_{R_\beta\times\prod_{\gamma\neq\beta}\{0_\gamma\}}=(\prod_{\alpha\in\Omega}\mathcal T_\alpha)\upharpoonright_{R_\beta\times\prod_{\gamma\neq\beta}\{0_\gamma\}}$ by minimality of $(R_\beta,\mathcal T_\beta)$.

Then the family $\{V\times\prod_{\gamma\neq\beta}\{0_\gamma\}\}$ when $V$ runs all neighborhoods of zero of $(R_\beta,\mathcal T_\beta)$ is  a fundamental system of neighborhoods of zero of
\[
\Big(R_\beta\times\prod_{\gamma\neq\beta}\{0_\gamma\},\quad \mathcal U\upharpoonright_{R_\beta\times\prod_{\gamma\neq\beta}\{0_\gamma\}}\Big).
\]
Since $R_\beta\times\prod_{\gamma\neq\beta}\{0_\gamma\}$ is an ideal with identity of $\prod_{\alpha\in \Omega}R_{\alpha}$, the topological ring $(\prod_{\alpha\in \Omega}R_{\alpha},\mathcal U)$ is a direct sum of ideals $R_\beta\times\prod_{\gamma\neq\beta}\{0_\gamma\}$ and $\{0_\beta\}\times\prod_{\gamma\neq\beta}R_\gamma$.
Let  $V$ be a neighborhood of zero of $(R_\beta,\mathcal T_\beta)$.
Then $V\times \prod_{\gamma\neq\beta}R_\gamma$ will be a neighborhood of zero of $(\prod_{\alpha\in \Omega}R_{\alpha},\mathcal U)$ and $\pi_\beta(V\times \prod_{\gamma\neq\beta}R_\gamma)=V$.

We have proved that $\pi_\beta$ is a continuous function from $(\prod_{\alpha\in \Omega}R_{\alpha},\mathcal U)$ to $(R_\beta,\mathcal T_\beta)$. It follows that  $\prod_{\alpha\in \Omega}\mathcal T_\alpha\leq \mathcal U$ and so $\mathcal U=\prod_{\alpha\in \Omega}\mathcal T_\alpha.$
\end{proof}

\begin{corollary}\label{C:2}
A left linearly compact semisimple ring is minimal if and only if has no direct summands of the form $M(n,\Delta)$, where $\Delta$ is a division ring which does not admit a nondiscrete Hausdorff ring topology.
\end{corollary}
\begin{proof}
This follows from Theorems \ref{T:6}, \ref{T:7} and the Theorem of Leptin (see \cite{Leptin}, Theorem 13, p. 258) about the structure of  left linearly compact semisimple rings.
\end{proof}

\begin{corollary}\label{C:3}
A semisimple linearly compact ring $(R, \mathcal{T})$ having no ideals isomorphic to matrix rings over infinite division rings is minimal.
\end{corollary}

%\newpage
\section{Completely simple endomorphism rings of modules}\label{P:3}

The  endomorphism ring of a right $R$-module $M$ is denoted by\\
 ${\rm End}\,(M_R)$.

\begin{lemma}\label{L:3}
Let $M$ be  a  divisible,  torsion-free module over a commutative domain $R$ and  $K$  the  field of fractions of $R$. The additive group of  $M$ has a structure of   a vector $K$-space such that $R$-endomorphisms of $M$ are exactly the  $K$-linear transformations.
\end{lemma}
\begin{proof}
We define a structure of a right vector $K$-space as follows: if  $\frac{a}{b}\in K$ and $m\in M$, then there exists a unique $x\in M$ such that $ma=xb$; set $m\circ \frac{a}{b} =x$.
Moreover,  if  $\frac{a}{b}=\frac{c}{d}$ and $0\neq m\in M$, then $m \circ \frac{a}{b}=m \circ \frac{c}{d}$. Indeed, if $m \circ \frac{a}{b}=x$ and $m \circ \frac{c}{d}=y$, then $mad=xbd$ and $mbc=ybd$ which means that $xbd=ybd$, hence $x=y$.

Let $\alpha\in {\rm End}\,(M_R)$, $\frac{a}{b}\in K$, $m\in M$. By definition, $am=b(\frac{a}{b}\circ m)$, hence, $a\alpha(m)= b\alpha(\frac{a}{b}\circ m)$, which means that $\alpha(\frac{a}{b}\circ m)=\frac{a}{b}\circ\alpha(m)$, so $\alpha$ is a $K$-linear transformation.
Note that,  if $a\in R$ and $m\in M$, then $m \circ \frac{a}{1}=ma$.

Conversely, if $\alpha$ is a $K$-linear transformation, $a\in R$, $m\in M$, then
$\alpha(\frac{a}{1}\circ m)=\frac{a}{1}\circ \alpha m$, i.e. $\alpha(am)=a\alpha(m)$. We have proved that every $K$-linear transformation is an right $R$-module homomorphism. \end{proof}

\begin{remark}\label{R:2}
The center $Z(R)$ of a weakly simple ring  $R$ is a domain.
\end{remark}
\begin{remark}\label{R:3}
For every right $R$-module $M$ the underlying group $M(+)$ is a discrete left topological   $({\rm End}\,(M_R),\mathcal{T}_{fin})$-module.
\end{remark}
Indeed, $T(m)(m)=0$ for every $m\in M$. Moreover,
\[
{\rm End}\,(M_R)\{0\}=\{0\},
\]
 so $M$ is a discrete left topological   $({\rm End}\,(M_R),\mathcal{T}_{fin})$-module.

\begin{theorem}\label{T:8}
Let $M_R$ be a module over a commutative ring $R$.

If the topological ring $({\rm End}\,(M_R),\mathcal{T}_{fin})$ is weakly simple, then:
\begin{itemize}
\item[{(i)}] $P=\{r\in R\mid Mr=0\}$ is a prime ideal of $R$.

\item[{(ii)}] $M$ is a vector space over the field $K$ of fractions of  $R/P$ and  the $R$-endomorphisms of $M$ are exactly the  $K$-linear transformations.

\end{itemize}

\noindent
Conversely,  if $M_R$ is an $R$-module and are satisfied (i) and (ii), then  $({\rm End}\,(M_R),\mathcal{T}_{fin})$ is a weakly simple topological ring.
\end{theorem}

\begin{proof}
$\Rightarrow:$
If  $({\rm End}\,(M_R),\mathcal{T}_{fin})$ is weakly simple, then  the mapping:
\begin{equation}\label{E:2}
\alpha_r: M \to M, \qquad  m\mapsto mr\qquad\quad(r\in R)
\end{equation}
is an $R$-module homomorphism and $\alpha_r\in Z$(= the center of ${\rm End}\,(M_R)$).

First we show that the part (i) holds.  Indeed, if $a,b\in R$ and $ab=0$, then $\alpha_a\alpha_b=0$ (see (\ref{E:2})). Thus $({\rm End}\,(M_R)\alpha_a)\cdot ({\rm End}\,(M_R)\alpha_b)=0$, so
\[
(\overline{{\rm End}\,(M_R)\alpha_a})\cdot (\overline{{\rm End}\,(M_R)\alpha_b})=0.
\]
Since ${\rm End}\,(M_R)$ is weakly simple, one of them,  say  $\overline{{\rm End}\,(M_R)\alpha_a}$, is zero.  This implies that $\alpha_a=0$, hence $a\in P$.

(ii) The structure of $R/P$-module  on $M$ is defined as follows:  if $r\in R$ and $m\in M$, then put $M(r+P)=mr$.

Note  that $M$ is a torsion-free right $R/P$-module. Assume  that  $m(r+P)=0$, where   $0\neq r+P\in R/P$ and  $0\neq m\in M$. Then $mr=0=\alpha_r(m)$ (see (\ref{E:2})).  Thus ${\rm End}\,(M_R)\alpha_r(m)=0$. It follows that  $\big(\overline{{\rm End}\,(M_R)\alpha_r}\big)(m)=0$ by  Remark \ref{R:3}. Since ${\rm End}\,(M_R)$ is weakly simple
\[
\overline{{\rm End}\,(M_R)\alpha_r}={\rm End}\,(M_R).
\]
We obtained that ${\rm End}\,(M_R)(m)=0$, so $m=0$, a contradiction.

 Under this convention  $R$-submodules are exactly $R/P$-submodules and $R$-endomorphisms are exactly
$R/P$-endomorphisms.

The module  $M$ is a divisible $R/P$-module. Indeed, if  $0\neq r+P\in R/P$, then  $0\not=M(r+P)=Mr$. Suppose that $Mr\neq M$. Consider
\[
I=\{\alpha\in E\vert \alpha(M)\subseteq Mr\}.
\]
Since $Mr$ is a fully invariant submodule, $I$ is a two-sided ideal of $({\rm End}\,(M_R),\mathcal{T}_{fin})$.

The ideal $I$ is closed. Indeed, let $\alpha\in \overline{I}$. If $m\in M$, then there exists $\beta\in I$ such that $\alpha-\beta\in T(m)$. Clearly,   $\alpha(m)=\beta(m)\in Mr$ and so $\alpha\in I$. We have proved that $I$ is closed.

Since $1_M\notin I$, $I=0$. It follows that    $\alpha_r=0$ (see (\ref{E:2})),  a contradiction.

The module $M$ has a structure of a right $K$-vector space and ${\rm End}\,(M_R)$ is exactly the ring of endomorphisms of $M$ by Lemma \ref{L:3}.

The converse follows from  Theorem \ref{T:6}.
\end{proof}

A characterization of completely simple topological ring ${\rm End}\,(M_R)$ is given by the following.

\begin{theorem}\label{T:9} Let $M_R$ be a module over a commutative ring $R.$
The topological  ring $({\rm End}\,(M_R), \mathcal{T}_{fin})$ is completely simple if and only  are satisfied the conditions (i) and (ii) of Theorem 8 and either
\begin{itemize}
\item[\rm(i)] $M$ is finite or
\item[\rm(ii)] $M$ is infinite and the dimension of $M$ over the field $K$  is infinite.
\end{itemize}
\end{theorem}

\begin{proof}
$\Rightarrow:$ According to Theorem \ref{T:8}, the ideal $P$ is prime   and the topology of  ${\rm End}\,(M_R)$ coincide with the finite topology of  ${\rm End}\,(M_K)$, where $K$ is the field of fractions of $R/P$. If $M$ is finite, we have the part (i).
Assume that $M$ is infinite. If $R/P$ is finite, then the dimension of $M$ over $K$ is infinite. Suppose that $R/P$ is infinite and $\dim _K(M)=n<\aleph_0$. Then $M$ is isomorphic to $M(n,K)$. Since $K$ is an infinite field, it admits a nondiscrete ring topology (see \cite{Hochster_Kiltinen}) and we obtain a contradiction because ${\rm End}\,(M_R)$ is a discrete ring. Consequently  $\dim _K(M)$ is infinite.

\noindent
$\Leftarrow$ This follows from Theorems \ref{T:6} and \ref{T:8}.\end{proof}

\begin{corollary}\label{C:4}
The topological ring $({\rm End}\,(A), \mathcal{T}_{fin})$ of an abelian group $A$ is completely  simple if and only  one of the following conditions holds:
\begin{itemize}
\item[\rm(i)] $A$ is a elementary abelian  $p$-group.
\item[\rm(ii)] $A$ is a divisible torsion-free group of infinite rank.
\end{itemize}
\end{corollary}

\newpage

\bibliographystyle{abbrv}
\bibliography{Bovdi_Ursul}
\end{document}